\documentclass{amsart}
\usepackage{amsmath,amsthm,amssymb}

\newtheorem{theorem}{Theorem}
\newtheorem{proposition}[theorem]{Proposition}
\newtheorem{lemma}[theorem]{Lemma}
\newtheorem{remark}[theorem]{Remark}
\newtheorem{example}[theorem]{Example}

\DeclareMathOperator{\card}{card}

\def\NN{\mathbb{N}}
\def\A{\mathcal{A}}
\def\C{\mathcal{C}}
\def\F{\mathcal{F}}
\begin{document}
\title[Hitting and returning into rare events]{Hitting and returning into rare events for all alpha-mixing processes}
\author{Miguel Abadi}
\email{miguel@ime.unicamp.br}
\address{Univ. Campinas, SP, Brazil}
\author{Benoit Saussol}
\email{benoit.saussol@univ-brest.fr}
\address{Univ. Bretagne Occidentale, Brest, France}
\begin{abstract}
We prove that for any $\alpha$-mixing stationnary process the hitting time of any $n$-string $A_n$ converges, when suitably normalized, to an exponential law. We identify the normalization constant $\lambda(A_n)$. A similar statement holds also for the return time.

To establish this result we prove two other results of independent interest.
First, we show a relation between the rescaled hitting time and the rescaled return time, generalizing a theorem by Haydn, Lacroix and Vaienti.
Second, we show that for positive entropy systems, the probability of observing any $n$-string in $n$ consecutive observations, goes to zero as $n$ goes to infinity.
\end{abstract}

\thanks{This work was part of the CAPES-COFECUB program Ma 545/07 which the authors greatly acknowledge.}
\maketitle

\section{Introduction}

The study of the statistical properties of the time elapsed until the occurrence of an observable of positive measure 
in a stationnary stochastic process and/or in a measure preserving dynamical system is a classical subject.
The starting point of this study is the famous Poincar\'e Recurrence Theorem who states that in an ergodic system, 
any set of positive measure appears in the process infinitely many times.
This is a qualitative result in the sense that no statistical properties of these returns are established.
In the last twenty years many notions of return were introduced and studied.
These notions depend on the initial conditions, the observed set, and on the measure of the system.
There was an intensive interest to study their statistical properties
to model physical phenomena like intermittence and metastability. Then, the applications  were extended to 
other areas such biology, linguistic and computer science to describe phenomena like
gene occurrence in a DNA and protein sequences, rhythm of a language and data compression algorithms, to mention
some of them.

In the present paper we consider a fixed set $A$ of positive measure $\mu(A)$ in an ergodic system.
When the evolution starts outside $A$, the time elapsed until the first occurrence of the set, is referred as the hitting time of $A$.
When the evolution starts inside $A$, the time is referred as the return time to $A$.

Our main result is that under the so called $\alpha$ or strongly mixing condition,
the distribution of the hitting time of a set $A$ can be well approximated by an exponential law.
The approximation is in the supremum norm in the space of distribution functions.
Although the exponential law is a classical subject our result is new and interesting:\\
a) Our results holds for \emph{any} cylinder set, namely, around any point, including periodic points and not just around generic points.\\
b) The result holds for any $\alpha$-mixing systems, while the best previous works \cite{aba06} assumed a polynomial rate of at least $(1+\sqrt{5})/2$.
Moreover, this strong-mixing condition is the weakest among many types of mixing conditions, among them $\psi, \phi, \rho, \beta$ or absolutely regular, $I$ or information regularly. See Bradley \cite{bra}.
\\
c) We also show that the exponential law holds when considering not just a cylinder set but even a set which is a union
of cylinders. Moreover, the cardinal of this union can be exponentially large, with respect to the length of the cylinders. 

Following the Galves and Schmitt \cite{gs} approach we get that the parameter of the exponential law is the product
$\lambda(A)\mu(A)$, where $\lambda(A)$ is a positive number related to the short recurrence properties of the set $A$.
For a description of these properties see Abadi \cite{aba06}.
In the aforementioned paper, the authors show that for $\psi$-mixing systems, there exist two positive constants $K, K'$ 
such that $K \le \lambda(A)\le K'$.
In our case, the constant $K$ does not exist, and one can have $\lambda(A)$ arbitrarily small. 

We prove our result by showing other two results which are interesting by themselves.
In the first one, we establish an ergodic relationship between the re-scaled hitting time $\lambda(A)\mu(A)\tau_A$
and the equally re-scaled return time.
The idea of this result comes from a paper of Haydn, Lacroix and Vaienti \cite{hlv}, which established
such a relationship for the rescaled $\mu(A)\tau_A$ hitting time and return time.
This in general does not apply in our case since one can have $\lambda(A)\not=1$, for instance, around periodic points.
The proof follows even a different approach.

The second result we mentioned above read as follows. The probability of observing an $n$-cylinder, or even a union of them,
in $n$ consecutive observations, goes to zero with $n$ for $\alpha$-mixing systems. Moreover, we show that the convergence is uniform on $A$.
It only depends on the cardinality of the union, but not in the choice of the cylinders.
This is natural when the measure of the set decays e.g. exponentially with $n$. But is far from obvious and maybe even anti-intuitive, when the measure decays just polynomially fast with power less than one, as it is covered by our case.

\section{Statement of the results}

Let $\A$ be a finite or countable set and let $\Sigma=\A^\NN$ be the set of sequences.
We endow $\Sigma$ with the shift map $T$. Given non negative integers $m\le n$ and a point $x\in \Sigma$ we denote by $[x_m\ldots x_n]$ the cylinder of rank $(m,n)$ containing $x$, that is 
\[
[x_m\ldots x_n] := \{y\in\Sigma\colon y_m=x_m,\ldots,y_n=x_n\}.
\]
A cylinder of rank $(0,n-1)$ will be simply called of rank $n$.
We denote by $\C_m^n$ the collection of cylinders of rank $(m,n)$ and by
$\F_m^n$ the $\sigma$-algebra generated by the partition $\C_m^n$. 
Let $\F$ be the $\sigma$-algebra generated by the $\F_m^n$'s and $\mu$ be a $T$-invariant probability measure on $(\Sigma,\F)$. 
Let 
\[
\alpha(g) =\sup_{m, n}\sup_{A\in\F_0^n,B\in\F_{n+g}^{m+g}}\left|\mu(A\cap B)-\mu(A)\mu(B)\right|
\]
for any integer $g$. We assume that the system $(\Sigma,T,\mu)$ is $\alpha$-mixing, in the sense that $\alpha(g)\to0$ as $g\to\infty$. 
This is the weakest notion of mixing among $\phi$ and $\psi$-mixing. We emphasize that we do not assume any summability condition on the sequence $\alpha(g)$.

Let $A\in\Sigma$ be a measurable set. We define the hitting time to $A$ by 
\[
\tau_A(x)=\inf\{k\ge1\colon T^kx\in A\}, \quad x\in\Sigma.
\]
We are interested in the distribution of the hitting time $\tau_A$ on the probability space $(\Sigma,\mu)$, and the return time, defined with the same formula but on the probability space $(A,\mu(\cdot|A))$ where $\mu(\cdot|A)$ denotes the conditional measure on $A$.

\begin{theorem}\label{hitting}
Suppose that the system $(\Sigma,T,\mu)$ is $\alpha$-mixing. 
Then for any sequence $A_n\in\F_0^{n-1}$ such that 
\begin{equation}\label{eq:tauzero}
\mu(\tau_{A_n}\le n)\to0\quad \text{as }n\to\infty,
\end{equation}
there exists some normalizing constant $\lambda(A_n)>0$ such that the following holds:

\begin{itemize}
\item
the hitting time to $A_n$, rescaled by $\lambda(A_n)\mu(A_n)$, converges in distribution to an exponential distribution. Namely, 
\[
\sup_{t\ge0} \left|\mu(\lambda(A_n)\mu(A_n)\tau_{A_n}>t)-\exp(-t)\right|\to 0
\text{ as } n\to\infty.
\]
The convergence is uniform on families of sets $A_n$ where the convergence in \eqref{eq:tauzero} is uniform.
\item
the distribution of the return time is approximated by a convex combination of a Dirac mass at zero and an exponential distribution. More precisely,
\[
\sup_{t\ge s}\left|\lambda(A_n)^{-1}\mu(\lambda(A_n)\mu(A_n)\tau_{A_n}> t|A_n)-\exp(-t)\right|\to0
\text{ as } n\to\infty,
\]
for any $s>0$.
\item 
we have $\limsup\lambda(A_n)\le1$.
\end{itemize}
\end{theorem}
The normalizing constant $\lambda(A_n)$ may not converge in general, thus we cannot simply say that the limiting distribution of the rescaled return time exists.
Moreover, even if it converges the limit may not be equal to one. For example a case of interest is when $\lim\lambda(A_n)=0$ where we still get a non-trivial exponential approximation, while without the extra factor $\lambda(A_n)$ one would just obtain the rough statement that the rescaled hitting time $\mu(A_n)\tau_{A_n}\to +\infty$ and the rescaled return time $\mu(A_n)\tau_{A_n}\to0$ in distribution.

In the next section we show that the hypothesis in the theorem holds for a broad class of sequences of sets $A_n$.

\section{Rare events do not appear too soon}

We present some explicit examples of sequences $A_n$ under which Theorem~\ref{hitting} applies, that is when the condition~\eqref{eq:tauzero} of the theorem is satisfied. They are consequences of Proposition~\ref{pro:tauAn} presented below.

The first example was the motivation of our work:
\begin{example}\label{ex:cyl}
For any $a\in\A^\NN$, the sequence of cylinders $A_n=[a_0,\ldots,a_{n-1}]$ satisfies the hypothesis of Theorem~\ref{hitting}. Moreover, the convergence is uniform on $a$.
\end{example}
We emphasize that this approximation with an exponential distribution is valid for any point $a\in\Sigma$, including for example periodic points. This generalizes the result in \cite{gk} which concern a.e. sequence $a$.

Returns to the cylinder $[a_0,\ldots,a_{n-1}]$ in the example above means that there is a perfect matching of the first $n$ symbols. It turns out that for some applications the approximate matching is more interesting:

\begin{example} Approximate matching:
Let $a\in\A^\NN$ and $D\in(0,1)$. Denote for $b\in\Sigma$ by $d_n(a,b)=\card\{i\le n-1\colon a_i\neq b_i\}$ the Hamming distance of the first $n$ symbols. Let 
\[
A_n= \{b\in\Sigma\colon d_n(a,b)\le Dn\},
\]
be the $D\%$ approximate matching of $[a_0,\ldots,a_{n-1}]$. Then there exists $D_0>0$ such that for all $D\in (0,D_0)$, the sequence $A_n$ satisfies the hypothesis of Theorem~\ref{hitting}.

In DNA sequence analysis the alphabet $\A$ is $\{A,C,G,T\}$. For some sequences the entropy is lower estimated by $1.7$ bits per symbol (for example the human gene {\sc Humretblas}; see \cite{acgt}), which means $h_\mu=1.7\ln2$ . This gives a value of $D_0\approx 41\%$.
\end{example}

\begin{proof}
We count the number $\kappa_n$ of cylinders of rank $n$ which compose the $D\%$ approximate matching $A_n$. We have 
\[
\begin{split}
\kappa_n 
&\le \sum_{k=0}^{Dn} {n \choose k} (\card\A-1)^k \\
&\le D^{-Dn}\sum_{k=0}^{n} {n \choose k} D^k(\card\A-1)^k \\
&= \left(\frac{1+D(\card\A-1)}{D^D}\right)^n
\end{split}
\]
We choose $D_0>0$ as the smallest solution of $(1+D(\card\A-1))/D^D=e^{h_\mu(T)}$ and then Proposition~\ref{pro:tauAn} applies for any $D<D_0$.
\end{proof}

\begin{example}\label{ex:topo}
For a set $K\subset \Sigma$ define its topological entropy by
\[
h_{\mathrm{top}}(K)=\limsup_{n\to\infty}\frac1n \log \#\{C\colon C\text{ cylinder of rank $n$ s.t. } C\cap K\neq\emptyset\}.
\]
Denote by $\F_0^{n-1}(K)$ the union of those cylinders $C$ of rank $n$ such that $K\cap C\neq\emptyset$.
The sequence $A_n=\F_0^{n-1}(K)$, under the assumption that $h_{\mathrm{top}}(K)<h_\mu$, satisfies the hypothesis of Theorem~\ref{hitting}.
\end{example}

\begin{example}
Suppose that $A_n=A_n^0\cup A_n^1$ where $A_n^0$ and $A_n^1$ are $\F_0^{n-1}$ measurable sets and such that $\lim n\mu(A_n^0)=0$ and $A_n^1$ satisfies the conditions of Example~\ref{ex:topo} above. 
Then, we have
\[
\mu(\tau_{A_n}\le n)\le \mu(\tau_{A_n^0}\le n)+\mu(\tau_{A_n^1}\le n)\le n\mu(A_n^0)+\mu(\tau_{A_n^1}\le n)\to0,
\]
therefore the hypothesis of Theorem~\ref{hitting} is satisfied.
\end{example}
We emphasize that, in this example, the exponential growth of the number of $n$-cylinders inside $A_n$ is not a priori bounded by the entropy of the measure, contrary to the preceding example. 

\begin{proposition}\label{pro:tauAn}
Suppose that $(\Sigma,T,\mu)$ is an ergodic measure preserving system, not necessarily $\alpha$-mixing.
Let $(\kappa_n)$ be a sequence of integers such that 
\[
\limsup_n\frac1n\log \kappa_n<h_\mu(T).
\] 
Then there exists a sequence $\epsilon_n\to0$ such that, for any $A_n\in\F_0^{n-1}$ which is the union of at most $\kappa_n$ cylinders of rank $n$ we have 
\[
\mu(\tau_{A_n}\le n)\le\epsilon_n.
\]
\end{proposition}
We emphasize that the bound $\epsilon_n$ does not depend on the particular set $A_n$ but only on the number of cylinders which compose it. Note that the statement $\mu(\tau_{A_n}\le n)\to0$ is trivial whenever $\mu(A_n)\ll 1/n$. However, even for  $\alpha$-mixing systems, there can exist some cylinders $A_n$ of rank $n$ such that $\mu(A_n)\gg 1/n$ (See~\cite{liliam}).

When the system is $\alpha$-mixing, the measure preserving transformation $(T,\mu)$ is an exact endomorphism and in particular, its entropy $h_\mu(T)$ is positive (we refer to \cite{CFS} for details). In particular Proposition~\ref{pro:tauAn} applies under the mixing hypotheses of Theorem~\ref{hitting}.

\begin{proof}

Set $h_0:=\limsup_n\frac1n\log \kappa_n$ and let $h\in(h_0,h_\mu(T))$ and $k\in\NN$ such that $h_0<(1-1/k)h$.
Let 
\[
\Gamma(N)=\{x\colon \forall n\ge N, \mu([x_0\ldots x_{n-1}])\le e^{-nh}\}.
\]
By the Shannon-McMillan-Breiman theorem $\mu(\Gamma(N))\to1$ as $N\to\infty$.
Given an integer $n$, let $m=\lceil n/k \rceil$ be the smallest integer such that $km\ge n$.
First, observe that by invariance we have
\begin{equation}\label{eq:An}
\mu(\tau_{A_n}\le n)\le k \mu(\tau_{A_n}\le m).
\end{equation}
Let $\displaystyle U_n=\bigcup_{j=0}^{m-1}T^jA_n$. We have $\{\tau_{A_n}\le m\}\subset T^{-m}U_n$,
hence 
\begin{equation}\label{eq:Un}
\mu(\tau_{A_n}\le m) \le \mu(U_n).
\end{equation}
Moreover, since each $T^j A_n$ is contained in an union of at most $\kappa_n$ cylinders of rank $n-j$, the set $U_n$ is contained in at most $m\kappa_n$ cylinders of rank $n-m$
therefore 
\[
\mu(U_n\cap \Gamma(n-m))\le m \kappa_n e^{-(n-m)h}.
\]
On the other hand,
\[
\mu(U_n\setminus \Gamma(n-m))\le 1-\mu(\Gamma(n-m)).
\]
Setting $\epsilon_n$ equal to $k$ times the sum of the last two upper bounds proves the proposition in view of~\eqref{eq:An} and~\eqref{eq:Un}.
\end{proof}

\section{Proof of the main theorem}

Our main theorem will be a direct application of this explicit estimation of the difference between the hitting time statistics and the exponential distribution.

\begin{theorem}\label{pro:main}
Suppose that the system $(\Sigma,T,\mu)$ is $\alpha$-mixing.
Let $n$ be an integer. For any $A\in\F_0^{n-1}$  there exists some constant $\lambda(A)\in(0,2]$ such that 
\[
\sup_{k\in\NN}\left|\mu(\tau_A>k)-e^{-\lambda(A)\mu(A)k}\right|\le 12\sqrt{2\mu(\tau_{A}\le n)+\alpha(n)}.
\]
\end{theorem}
The value of the upper bound is not intented to be optimal, but is just there to emphasize that it does not depend on the particular choice of the set $A\in\F_0^{n-1}$ but only on the probability of short hitting times $\mu(\tau_{A}\le n)$.

In the proof of the theorem we make use of the following lemma.
\begin{lemma} \label{lem:scale}
Let $n$ be an integer. For any $A\in\F_0^{n-1}$ such that 
\[
\delta := 3\sqrt{2\mu(\tau_{A}\le n)+\alpha(n)}<1/4,
\]
there exist an integer $s> 2n$ such that 
\begin{equation}\label{eq:sn}
\mu(\tau_{A}\le s)\le\delta  \qquad {\rm and} \qquad 
\frac{\mu(\tau_{A}\le 2n)+\alpha(n) }{ \mu(\tau_{A}\le s-2n) } \le \delta.
\end{equation}
\end{lemma}

\begin{proof} 
Let us define $d=2\mu(\tau_{A}\le n)+\alpha(n)$. By hypothesis $d<1/144$.
By invariance we have 
\[
\mu(\tau_{A}\le 2n)+\alpha(n)\le d.
\]
Let $s> 2n$ denotes the smallest integer such that 
\[
 \mu(\tau_{A} \le s-2n)\ge\sqrt{d}.
\]
With this choice we have
\[
\frac{\mu(\tau_{A}\le 2n)+\alpha(n) }{ \mu(\tau_{A}\le s-2n) } \le \sqrt{d}.
\]
Furthermore, since $\mu(\tau_{A} \le s-2n-1)<\sqrt{d}$, it follows from the invariance that
\[
\mu(\tau_{A}\le s)\le \mu(\tau_{A}\le s-2n-1)+\mu(\tau_{A}\le 2n+1)\le \sqrt{d}+2d.
\]
\end{proof}

\begin{proof}[Proof of Theorem~\ref{pro:main}]
Let $n$ be an integer and $A\in\F_0^{n-1}$. 
Let $\delta$ be as in Lemma~\ref{lem:scale}. There is nothing to prove if $\delta\ge 1/4$ so we suppose that $\delta<1/4$.
Take $s>2n$ given by Lemma~\ref{lem:scale} such that~\eqref{eq:sn} holds. 

To simplify notation we drop the subscript $A$ and write $\tau=\tau_A$. Set $H(k)=\mu(\tau>k)$,
and denote by $\tau^{[t]}=\tau\circ T^{t}$ the first occurrence time starting at time $t$.
For any integer $j\ge 1$ consider the modulus
\begin{equation}\label{eq:HA}
| H( j s ) - H( (j-1)s ) H( s - 2n ) | .
\end{equation}
The sets 
\[
\{\tau>js\}=\{\tau>(j-1)s\}\cap\{\tau^{[(j-1)s]}>s\}
\]
and 
\[
\{\tau>(j-1)s\}\cap\{\tau^{[(j-1)s+2n]}>s-2n\}
\] 
differ by a subset of $\{\tau^{[(j-1)s]}\le 2n\}$ whose measure is by invariance bounded by $\mu(\tau\le 2n)$.
Furthermore, by mixing we get that
\[
|\mu(\{\tau>(j-1)s ; \tau^{[(j-1)s+2n]}>s-2n\}) - H( (j-1)s ) H( s - 2n )|\le \alpha(n).
\]
Thus the above expression~\eqref{eq:HA} is bounded by
\[
\mu(\tau\le 2n) + \alpha(n)
\]

Now, take $q$ a positive integer.
The absolute value 
\begin{equation} \label{integers}
| H( qs  ) -  H(s-2n)^q  |
\end{equation}
is bounded by
\[
\sum_{j=1}^{q}
| H(j s  ) - H( (j-1)s  ) 
  H( s- 2n  ) | H( s- 2n  ) ^{q-j} .
\]
We just proved that the modulus in the above sum is bounded by $\mu(\tau\le 2n)+\alpha(n)$.
Summing over $j$ we get that for all integer $k\ge 1$ the modulus in (\ref{integers}) is bounded by
\[
\frac{ \mu(\tau \le 2n)+\alpha(n)  }{\mu(\tau\le s-2n) } \le\delta.
\]

Moreover,  any non-negative integer $k$  can be written as $q s+r$ with $q=[k/s]$ and $0\le r < s$.
Then
\begin{equation}\label{eq:ketqs}
|H(k)-H(q s)|=\mu(\tau>q s ; \tau^{[q s]} \le r   ) ,
\end{equation}
which, by invariance, is bounded by $\mu(\tau\le s)\le\delta$.

To finish the proof, set
\[
\lambda(A) = - \frac{\ln H( s-2n)}{s\mu(A)},
\]
and note that
the Mean  Value Theorem gives
\begin{equation}\label{eq:ksurs}
|H( s-2n)^{[k/s]}-H( s-2n)^{k/s}| \le 
-\ln H(s-2n).
\end{equation}
Note that 
$H( s-2n)^{k/s}=e^{-\lambda(A)\mu(A) k } $. By convexity we have $-\ln(1-u)\le u / (1-\delta)$ whenever $0\le u\le \delta$, therefore
\[
-\ln H(s-2n) \le \frac{1}{1-\delta}\mu(\tau \le s-2n)\le\frac{1}{1-\delta}\mu(\tau\le s)\le2\delta.
\]
Putting together the three estimates for \eqref{integers}, \eqref{eq:ketqs} and \eqref{eq:ksurs} gives the conclusion.
Observe in addition that $\lambda(A)\le 1/(1-\delta)\le2$ since $\mu(\tau\le s)\le s\mu(A)$.
\end{proof}
\begin{remark}\label{rem:to1}
The upper bound $\lambda(A)\le2$ can be sharpened when $\delta$ is small. In particular  if $\delta_n\to0$ as $n\to\infty$ we get $\limsup\lambda(A_n)\le 1$.
\end{remark}

We conclude this section with the proof of the main theorem.
In view of Theorem~\ref{thm:FandG}, the statement for hitting times in the main theorem (Theorem~\ref{hitting}) and the one for return times are equivalent, hence it is sufficient to prove the first statement with $F(t)=1-e^{-t}$, which will imply the second statement with $G(s)=e^{-s}$.

\begin{proof}[Proof of Theorem~\ref{hitting}]
For any real $t>0$, taking $k=\lfloor t/\mu(A_n)\rfloor$ in Theorem~\ref{pro:main} gives
\[
\left|\mu(\lambda(A_n)\mu(A_n)\tau_{A_n}>t)-e^{-t}\right|\le 12\sqrt{2\mu(\tau_{A_n}\le n)+\alpha(n)}+2\mu(A_n),
\]
which proves the first statement. The uniform convergence in \eqref{eq:tauzero} implies that of this upper bound, since
\[
\mu(A_n)=\mu(\tau_{A_n}=1)\le\mu(\tau_{A_n}\le n).
\]
The second statement follows from Theorem~\ref{thm:FandG}.
The third statement follows from Remark~\ref{rem:to1}.
\end{proof}

\section{Hitting and returning: an adaptation of haydn-Lacroix-Vaienti theorem}

Haydn, Lacroix and Vaienti \cite{hlv} have prove that the asymptotic distribution of hitting and return times $\tau_{A_n}$, rescaled by the measure $\mu(A_n)$ are related by an integral equation. Their result does not apply to our setting because the asymptotic distribution does not exist in general, because the normalizing constant does not converge in general. 

We now give the generalization of their result adapted to our case, which deserves a new proof since the technique needs to be relatively different.
Let 
\[
\begin{split}
F_A(t)&=\mu(\lambda(A)\mu(A)\tau_A \le t),\\
G_A(s)&=\frac{1}{\lambda(A)}\mu(\lambda(A)\mu(A)\tau_A> s|A).
\end{split}
\]
$F_A$ is the usual non-decrasing cumulative distribution function of the rescaled hitting time  $\lambda(A)\mu(A)\tau_A$ while $G_A$ is a normalized non-increasing distribution function of the rescaled return time $\lambda(A)\mu(A)\tau_A$.
We recall that since $F_A$ and $G_A$ are monotonous, their convergence when $\mu(A)\to0$ on a dense set or on all but countably many points are equivalent and we will simply say that they converge.

\begin{theorem}\label{thm:FandG}
Suppose that the measure preserving system $(\Sigma,T,\mu)$ is ergodic.

Let $A_n$ be a sequence of measurable sets such that $\mu(A_n)\to0$.
If $F_{A_n}$ converges to $F$ as $n\to\infty$ then $G_{A_n}$ converges to some function $G$, and the limits are related by the integral equation
\[
F(t) = F(0+) + \int_0^tG(s)ds\quad (\forall t>0).
\]
In particular, if the solution $G$ is continuous then the convergence is uniform on $[s,+\infty)$ for any $s>0$.

Reciprocally, if $G_{A_n}$ converges to $G$ as $n\to\infty$ and $\int_0^\infty G(s)ds=1$ then $F_{A_n}$ converges to some function $F$, and the limits are related by the same integral relation with $F(0+)=0$.
In particular, $F$ is continuous on $[0,\infty)$ and the convergence is uniform.
\end{theorem}
\begin{proof}
Let $A$ be any measurable set with $\mu(A)>0$.
Note that $0\le F_A(t)\le 1$ and $0\le G_A(s)\le 1/s$ for any $s>0$, where
this last upper bound follows from Markov inequality and Kac's Lemma:
\[
G_A(s) = \frac{1}{\lambda(A)}\mu(\lambda(A)\mu(A)\tau_A>s|A)\le \frac{\mu(A)}{s}\int\tau_A d\mu(\cdot|A)\le\frac{1}{s}.
\]
First observe that by invariance one has for every integer $n$
\[
\mu(\tau_A=n)=\mu(A\cap \{\tau_A\ge n\}).
\]
Therefore
\[
\begin{split}
F_A(t)&=\sum_{n=1}^{t/\lambda(A)\mu(A)}\mu(A\cap \{\tau_A\ge n\})\\
&=\int_0^{\lfloor t/\lambda(A)\mu(A)\rfloor}\mu(A\cap \{\tau_A> r\})dr.
\end{split}
\]
Since $\mu(A\cap\{\tau_A>r\})\le\mu(A)$ we get by a change of variable
\[
F_A(t)\le \int_0^t G_A(s)ds \le F_A(t)+\mu(A).
\]
For any $0<t<t'$ we get the relation
\begin{equation}\label{eq:FAGA}
\int_t^{t'}G_A(s)ds -\mu(A) \le F_A(t')-F_A(t)\le \int_t^{t'}G_A(s)ds +\mu(A).
\end{equation}

$\bullet$
Assume that $F_{A_n}$ converges to some function $F$ and suppose for a contradiction that $G_{A_n}$ does not converge.
By Helly's selection principle, each subsequence of function must have an accumulation point\footnote{Indeed, the space of decreasing functions $g$ from $(0,\infty)$ to itself such that $g(s)\le s$, under the equivalence relation of equality outside countable sets, is metrizable (e.g. a slight modification of the Levy metric) and compact (Helly selection principle) and an accumulation point refers to this notion of convergence.}. Therefore $G_{A_n}$ must have at least two different accumulation points $G_1$ and $G_2$.
By dominated convergence~\eqref{eq:FAGA} gives that for all $0<t<t'$
\begin{equation}\label{eq:FGi}
F(t')-F(t)=\int_t^{t'}G_i(s)ds\quad(i=1,2)
\end{equation}
Hence $G_1=G_2$ a.e., a contradiction; thus $G_{A_n}$ converges.
Lastly, the integral relation follows from~\eqref{eq:FGi} by monotone convergence.

$\bullet$
Assume that $G_{A_n}$ converges to some function $G$. 
By Fatou's lemma, the left-most inequality in~\eqref{eq:FAGA} gives that for all $t>0$
\[\label{eq:Finf}
\int_0^{t}G(s)ds\le\liminf_{n\to\infty} F_{A_n}(t)
\quad;\quad
\int_t^\infty G(s)ds\le\liminf_{n\to\infty} \left( 1-F_{A_n}(t)\right).
\]
therefore under our assumption on the limit $G$, $F_{A_n}$ converges to $F$ and 
\[
F(t)=\int_0^tG(s)ds.
\]
\end{proof}

\end{document}